\documentclass{article}

\usepackage{amssymb,latexsym}
\usepackage{amsmath, amsthm}
\usepackage{verbatim}
\usepackage[mathscr]{eucal}

\usepackage{enumerate}
\makeatletter
\@namedef{subjclassname@2010}{%
  \textup{2010} Mathematics Subject Classification}
\makeatother

\newtheorem{theorem}{Theorem}[section]
\newtheorem{corollary}[theorem]{Corollary}
\newtheorem{proposition}[theorem]{Proposition}
\newtheorem{lemma}[theorem]{Lemma}

\theoremstyle{definition}
\newtheorem{definition}[theorem]{Definition}
\newtheorem{example}[theorem]{Example}

\numberwithin{equation}{section}

\newcommand{\norm}{\left\Vert\,\cdot\,\right\Vert}
\newcommand {\C}{\mathbb C}

 \newcommand {\N}{\mathbb N}
\newcommand {\Z}{\mathbb Z}
\newcommand {\F}{\mathbb F}

\newcommand{\lv}{\left\vert}
\newcommand{\rv}{\right\vert}
\newcommand{\lV}{\left\Vert}
\newcommand{\rV}{\right\Vert}

\newcommand{\bb}{{\,\scriptstyle {\Box}}\,}

\newcommand{\s}{\smallskip}

\def\sharp{\#}

\def\supp{\mathop{\rm supp\,}}

\def\notd{{\raise.03em\hbox{{\hbox{\it\scriptsize /}}}\kern-.575em\mid}\,}

\title{Radicals of some semigroup algebras}

\author{H.~G.~Dales, D.~Strauss, Y.~Zelenyuk, and Yu.~Zelenyuk}

\begin{document} 
\maketitle

\begin{abstract}
 In this paper we seek to determine the Jacobson radical of certain algebras based on semigroups, and in particular  on the semigroups 
$(\beta S, \Box)$, where $S$ is a cancellative, countable, abelian semigroup and $\beta S$ is its  Stone--\v{C}ech semigroup com\-pactification. In particular, we wish to determine the radical of $\ell^{\,1}(\beta \N)$.
\end{abstract}
 
\section{Introduction} 

\noindent We first recall the basic definitions and properties of the (Jacobson) radical of an algebra; see \cite{D} for details.
 
 Let $A$ be a (complex, associative)  algebra.   The algebra formed by adjoining an identity to $A$ is denoted by $A^{\sharp}$, with $  A^{\sharp}= A$ when 
$A$ already has an identity. The  identity of $A^{\sharp}$ is denoted by  $e_A$. The {\it radical\/} of $A$, denoted by  $J(A)$, is defined to be the
 intersection of the maximal left ideals of $A^{\sharp}$; 
it is also equal to the intersection of the maximal right ideals of $A^{\sharp}$, and so it is an ideal in $A$. The algebra $A$ is {\it semisimple\/} 
if $J(A)=\{0\}$; the quotient algebra $A/J(A)$ is a semisimple algebra.

Let $A$ be an algebra.   An element $a\in A$ is {\it nilpotent} if $a^n= 0$  for some $n\in\N$; the minimum such $n$ is the {\it index\/} of $a$;  
the set of nilpotent elements of $A$ is denoted by ${\mathcal N}(A)$.   An element $a\in A$ 
is {\it quasi-nilpotent} if  $ze_A-a$ is invertible in $A^{\sharp}$ for each $z \in \C$ with $z\neq 0$;  the set of quasi-nilpotent elements of $A$ is denoted by
 ${\mathcal Q}(A)$.   Trivially,  ${\mathcal N}(A) \subset {\mathcal Q}(A)$.
 
A characterization of $J(A)$ is as follows \cite[Proposition 1.5.32(ii)]{D}.  
 
 \begin{theorem}  \label{1.1}
Let $A$ be an algebra. Then
\[
J(A) = \{a \in A : ba \in {\mathcal Q}(A)\;\; (b\in A^{\sharp})\} = \{a \in A : ab \in {\mathcal Q}(A)\;\; (b\in A^{\sharp})\}\,.
\qedhere \]
\end{theorem}\s
 
Thus, for $a\in A$, we have $a \in J(A)$ if and only if, for each $b \in A$, there exists $c \in A$ with $ba + c = cba$.  

 It follows from Theorem \ref{1.1} that   $J(A) \subset {\mathcal Q}(A)$.
 In general, $J(A) \subsetneq {\mathcal Q}(A)$, and neither ${\mathcal N}(A)$ nor ${\mathcal Q}(A)$ is closed under either sums or products; 
this is shown by simple examples of $2 \times 2$ matrices. 

We shall use the following standard result;      clauses 
(i) and (ii) are contained in   \cite[Theorem 1.5.4]{D}, and clause (iii)  follows from \cite[Corollary 1.5.3(ii)]{D}.

\begin{proposition} \label{1.2}
  Let $A$ be an algebra, and let $I$ be an ideal in $A$.\s
 
{\rm (i)}    $J(I)= J(A)\cap I$.\s

{\rm (ii)}  Suppose that   $I\subset J(A)$.  Then  $J(A/I) = J(A)/I$. \s

{\rm (iii)}  Suppose that  $A/I$ semisimple. Then $J(A) \subset I$.\quad\qedhere 
\end{proposition}

 Now let $A$ be a Banach   algebra.  Then $J(A)$ is a closed ideal in $A$, and $A/J(A)$ is a semisimple Banach algebra.  The spectral radius of
 $a \in A$ is  denoted by $\nu_A$, and so
\[
{\mathcal Q}(A)= \{ a \in A: \nu_A(a) = 0\} = \{ a \in A: \sigma(a)=\{0\}\}\,,
\]
where $\sigma(a)$ is the spectrum of $a\in A$. By the spectral radius formula \cite[Theorem 2.3.8]{D},
\begin{equation}\label{(2.1)}
\nu_A(a) = \lim_{n\to \infty} \lV a^n\rV^{1/n}\,.
\end{equation}
Suppose that $B$ is a subalgebra of $A$. Then $\lim_{n\to \infty}\lV b^n\rV^{1/n} =0$ for each $b\in J(B)$. 
  In the case where the Banach algebra $A$ is commutative, we have ${\mathcal N}(A)\subset 
J(A)= {\mathcal Q}(A)$; ${\mathcal N}(A) $ is not necessarily either closed or dense in $J(A)$.
For non-commutative Banach algebras, it may be that   ${\mathcal Q}(A)$ is  not $\norm$-closed in  $A$ \cite[Example 2.3.15]{D}. 
 For a normed algebra $A$, it may be that $J(A)$ is not closed in $A$ \cite[$\S10$]{Dix}.\s

\begin{proposition} \label{1.2a}
 Let $A$ be a Banach  algebra,  and let $B$ be a closed  subalgebra of $A$. Then $J(A)\cap B \subset J(B)$.
\end{proposition}

\begin{proof}  Let $a \in J(A)\cap B$, and take $b\in B$. Then $\nu_A(ba) =0$ and so  $\nu_B(ba) =0$ by (\ref{(2.1)}).  Thus $ba \in {\mathcal Q}(B)$, and 
so $a \in J(B)$.
\end{proof}\s

\begin{proposition} \label{1.2b}
 Let $A$ be a Banach  algebra,  and let  $I$ be a closed, left ideal  in $A$. Suppose that there is an element $u\in I$ such that $au \neq 0$ and $ua\neq 0$
 whenever $a \in A\setminus \{0\}$.   Then $I$ is semisimple if and only if $A$ is semisimple.
 \end{proposition}

\begin{proof}   Suppose that $I$ is not semisimple, and take $a\in J(I)$ with $a\neq 0$.  Then $ua \in J(I)$.  Take $b\in A$. Then 
$bu \in I$, and so $bua \in {\mathcal Q}(I) \subset  {\mathcal Q}(A)$. Thus $ua \in J(A)$. Since $ua \neq 0$, $A$ is not semisimple.

Suppose that $A$ is not semisimple, and take $a\in J(A)$ with $a\neq 0$.  Then $au \in J(A)\cap I$.   Take $b \in I$. 
Then $bau \in {\mathcal Q}(A)\cap I= {\mathcal Q}(I)$, and so  $au \in J(I)$. Since $au \neq 0$, $I$ is not semisimple.
\end{proof}\s

In general, a closed subalgebra of a unital, semisimple Banach algebra is not necessarily semisimple. For example, let $A = {\mathbb M}_n$,
 the algebra of $n\times n$ matrices over $\C$, so that $A$ is a unital,  semisimple, finite-dimensional  Banach algebra, and let $B$ be the closed, 
unital subalgebra of upper-triangular matrices. Then  $J(B)$ consists of the matrices that are zero on the diagonal, and so $J(B)\neq\{0\}$ whenever
 $n\geq 2$.
There are also easy examples of commutative, radical Banach algebras  with a dense, semisimple subalgebra.\medskip

  \section{Semigroup algebras} 

\noindent Let $S$ be a non-empty set. Then $\ell^{\,1}(S)$ is the usual Banach space consisting of the functions $f \in \C^S$ such that 
\[
\lV f \rV = \sum_{s\in S}\lv f(s)\rv< \infty\,.
\]
For   an element $f \in \ell^{\,1}(S)$, the {\it support\/} of $f$ is
$
\supp f = \{s\in S  : f(s) \neq 0\}$.
Of course, $\supp f $ is always countable. The characteristic function of $\{s\}$ for an element $s \in S$ is denoted by $\delta_s$, and  a generic element of 
$\ell^{\,1}(S)$  is written as $\sum_{s\in S} f(s)\delta_s$. The linear space 
 spanned by the  functions $\delta_s$ is $\C S$; these are the elements of {\it finite support\/}.  Thus  $\C S$ is a dense subspace of $(\ell^{\,1}(S), \norm)$.
  
We shall consider algebras $\C S$ and $\ell^{\,1}(S)$ based on certain semigroups $S$.  We first recall some properties of semigroups; 
for a substantial study of semigroups, see \cite{HS} and \cite{How}.

Let $S$ be a semigroup, with product denoted by juxtaposition. An element $p\in S$ is an {\it idempotent\/} if $p^2= p\,$;
 the set of idempotents of the semigroup $S$ is denoted by $E(S)$.  For $s\in S$, we set $L_s(t) =st$ and $R_s(t)=ts$ for $t\in S$; for $T\subset S$, 
we write $sT =L_s(T)$ and $Ts=R_s(T)$.  
An element $s\in S$ is {\it cancellable\/} if both $L_s$ and $R_s$ are injective, and   $S$ is {\it cancellative\/} if each $s \in S$ is  
 cancellable;   $s\in S$ is {\it weakly cancellable\/} if   $\{u\in S : su=t\}$  and $\{u\in S : us=t\}$ are both finite for each $t\in S$, and   $S$ is 
{\it weakly cancellative\/} if each $s \in S$ is weakly cancellable.
A subset $T$ of a semigroup $S$ is a {\it left ideal\/} if $sT\subset T\,\;(s\in S)$, a {\it right ideal\/} if $Ts\subset T\,\;(s\in S)$, and an 
{\it ideal\/} if it is both a left and right ideal; ideals in $S$ are ordered by inclusion. 

A semigroup $S$ is abelian if $st =ts\,\;(s,t \in S)$, and in this case we usually denote the semigroup operation by `$+$'. Let $(S, +)$ be a cancellative,
  abelian semigroup.  Then there is an abelian group $(G,+)$ containing $S$ as a subsemigroup and such that each $x\in G$
 can be expressed as $x=s-t$ for some $s,t \in S$; $G$ is called the {\it group of quotients} of $S$.
  
 A semigroup $S$ with a topology $\tau$ is a {\it compact, right topological semigroup\/}  if $(S,\tau)$ is a compact (Hausdorff) space 
and the map $R_s$  is continuous for each $s\in S$. These important semigroups are studied in \cite{HS}.
 
Study of the semigroups that will concern us is based on the following structure theorem; it is stated in somewhat more generality than we require. 
See \cite[$\S2.2$]{HS} for a much more general version.\s

\begin{theorem}\label{str} 
Let $V$ be a compact, right topological semigroup. \smallskip

 {\rm (i)} A unique minimum ideal $K(V)$ exists in $V$. The families of minimal left ideals and of minimal right ideals of $V$ both partition $K(V)$.\smallskip

{\rm (ii)} For  each  minimal right and left  ideals $R$  and $L$ in $V$, there exists $p\in E(V)\cap R \cap L$ such that   $R\cap L = RL = pVp$ is a  group; 
these groups are maximal in $K(V)$, are pairwise isomorphic, and the family of these groups partitions $K(V)$.\s

{\rm (iii)}   For each $p,q \in K(V)$, the subset $pK(V)q$ is a subgroup of $V$, and there exists $r \in E(K(V))$ with $rp=p$ and $qr=q$.\quad\qedhere  
\end{theorem} \s

Let $S$ be a semigroup. Then there is a unique product $\star$ on $\ell^{\,1}(S)$  such that 
\[\delta_s \,\star\,\delta_t=\delta_{st}\quad(s,t \in S)
\]
such that  
$(\ell^{\,1}(S), \,\star\,, \norm)$   is a Banach algebra; this is the {\it semigroup algebra\/} of $S$.  Thus, given $f,g\in \ell^{\,1}(S)$, we have
\[
(f\,\star\, g)(t) = \sum\{f(r)g(s) : r,s\in S,\,rs=t\}\quad (t\in S)\,,
\]
where the sum is zero when there are no elements $r,s\in S$ with $rs=t$.  The space $\C S$, the `algebraist's semigroup algebra', is 
a dense subalgebra  of our Banach algebra $\ell^{\,1}(S)$.  For $n\in\N$, the $n^{\rm th}$ power of $f\in \ell^{\,1}(S)$ is denoted by $f^{*\,n}$.

For an extensive study of this Banach algebra, see \cite{DLS1}. \s

\begin{definition} \label{2.1}
Let $S$ be a semigroup. The radical of the semigroup algebra $(\ell^{\,1}(S),\,\star\,)$ 
is denoted by  $J(S)$, and the sets of nilpotents and quasi-nilpotents in $\ell^{\,1}(S)$ are denoted by ${\mathcal N}(S)$  and ${\mathcal Q}(S)$, respectively.
The radical of the   algebra $(\C S,\,\star\,)$  is denoted by  $J_0(S)$.  
\end{definition}\s

 Let  $S$ be a semigroup. Then it follows from Theorem \ref{1.1} that
\[
 J(S) = \{f \in \ell^{\,1}(S): g\,\star f   \in {\mathcal Q}(S)\;\; (g\in \ell^{\,1}(S))\}\,.
\]

Easy examples show that there are finite, abelian semigroups $S$ such that  $\ell^{\,1}(S)$ is not semisimple. 
 For example, set $S = \{o,s\}$ where $o^2=os=so =s^2=o$, so that $S$ is an abelian semigroup (and $S$ is a zero semigroup).  
Then set $f = \delta_o -\delta_s$. Clearly $f$ is nilpotent of index 2 and  $\delta_o\,\star\,f$  and $\delta_s\,\star\,f$ are zero, and so 
$J(S) = \C f\neq\{0\}$.

 Let $S$ be a finite semigroup.  Then a criterion  for $\C S$ to be semisimple is  given in \cite[Chapter 14, Theorem 31]{Ok}. 
 
In the case where $S$ is an abelian semigroup, conditions for $\ell^{\,1}(S)$ to be semisimple are given in \cite{HZ}: indeed, $\ell^{\,1}(S)$ is
 semisimple if and only if  $S$ is {\it separating\/}, in the sense that $s=t$ whenever $s,t \in S$ and $s^2=t^2=st$.    

 In the case where $G$ is a group, $J(G)=\{0\}$, and so $\ell^{\,1}(G)$ is semisimple \cite[Corollary 3.3.35]{D}. It is also true that $J_0(G)=\{0\}$; this is a theorem of Rickart,  proved in \cite[Theorem 7.1.1]{Pass}, for example.  Further, ${\mathcal Q}(G)=\{0\}$
 for each abelian group $G$. Indeed, there is a standard, more-general theorem.  Let $G$ be a locally compact group with Haar measure $m$,
 and let $L^1(G, m)$ be the corresponding group algebra of $G$. Then $L^1(G, m)$ is a semisimple Banach algebra \cite[Corollary 3.3.35]{D}. 
 
Let $S$ be a cancellative semigroup. We do not know if $\ell^{\,1}(S)$ or $\C S$ is necessarily semisimple; this is true if $S$ is either finite or abelian. 
 It is not true that every cancellative semigroup is a subsemigroup  of a group \cite{Mal}; we do not know if $\ell^{\,1}(S)$  or $\C S$  
is necessarily semisimple whenever this is the case. Let ${\mathbb S}_n$ be the free semigroup on $n$ generators. Then it is true that 
$\ell^{\,1}({\mathbb S}_n)$ is semsimple: indeed,   $J({\mathbb S}_n)   = {\mathcal Q}({\mathbb S}_n)  = \{0\}$  \cite[Theorem 2.3.14]{D}. 
 For some partial results on when  $\C S$ is semisimple for particular cancellative semigroups, see  \cite[Chapter 10, Corollary 5 and Lemma 8]{Ok}.  
For example, each ordered semigroup $S$ is cancellative and such that $\C S$ is semisimple.

   We obtain the following corollary of Proposition \ref{1.2a}.\s

\begin{proposition}\label{2.20}
Let $S$ be a semigroup with a subgroup $G$. Suppose that $f \in J(S)$ with $\supp f \subset G$.
 Then $f=0$.\quad\qedhere 
\end{proposition}\s
 
\begin{proposition}\label{2.21}
Let $V$ be a compact, right topological semigroup, and suppose that  $p,q \in K(V)$. Take $f \in J_0( V)$  with $\supp f \subset pK(V)q$. Then $f=0$.
\end{proposition}

\begin{proof} Set $G= pK(V)q$, so that $G$ is a group.  By Theorem \ref{str}(iii), there exists $r \in E(V) \cap G$ such that $rp =p$ and $qr=q$  and such that
$rK(V)r =rVr$ is a group. 

Take $g \in \C\, G\subset \C S$. Then there exists $h\in \C S$ with $g\,\star\,f + h = h\,\star\,g\,\star\,f$. 
We have 
\[
g\,\star\,f + \delta_r\,\star\,h\,\star\,\delta_r = \delta_r\,\star\,h\,\star\,\delta_r\,\star\,g\,\star\,f 
\] 
because $\delta_r\,\star\,g = g$ and $f \,\star\,\delta_r  =f$, and $\supp(\delta_r\,\star\,h\,\star\,\delta_r) \subset rVr\subset G$, 
so  that  $\delta_r\,\star\,h\,\star\,\delta_r \in \C\, G$. Thus $f\in J_0(G) =\{0\}$. \end{proof}\s
 
Let $\F_2$ be the free group on two generators.  It is shown  in \cite[Lemma 7.3]{DLS1}  that there are nilpotent elements 
of every index in $\ell^{\,1}(\F_2)$ and that there are quasi-nilpotent elements that are not nilpotent.  Thus 
$\{0\}= J(\F_2) \subsetneq   {\mathcal N}(\F_2)   \subsetneq  {\mathcal Q}(\F_2)$.\medskip

\section{The semigroup $(S^*, \Box)$}

\noindent  The Stone--\v{C}ech compact\-ification 
of a discrete  topological space $S$ is denoted by $\beta S$;  we regard $S$ as a subset of  $\beta S$, and set $S^*=\beta S \setminus S$. More generally,
 we set $A^* = \overline{A}\cap S^*$ for a subset $A$ of $S$, where $\overline{A}$ is the closure of $A$ in $\beta S$.

Now, throughout this section, we take $S$ to be  a semigroup;   the particular example that we have in mind is $S = (\N, +)$. 
It is shown in many places, including \cite{DLS1, HS} 
(from different points of view), that, in the case where $S$ is a semigroup,  there is a unique binary operation $\Box$ on $\beta S$ such that $ (\beta S, \Box)$
 is a semigroup containing $S$ as a subsemigroup and such that $(\beta S, \Box)$ is a compact, right topological semigroup. \s
 
 \begin{definition}\label{3.1}
Let $S$ be a semigroup. Then the  semigroup  $(\beta S, \Box)$  is  the 
{\it Stone--\v{C}ech semigroup com\-pactification} of  $S$.
\end{definition}\s

In the  general case,  where the product in $S$ is denoted by juxtaposition, we shall usually denote the operation $\Box$ in $\beta S$ by juxtaposition and write just
 $\beta S$ for $(\beta S, \Box)$;   the corresponding  product in $\ell^{\,1}(\beta S)$ is denoted by $\,\star\,$.
In the special case where $S$ is abelian (and especially where $S= (\N,+)$), we shall sometimes write $(\beta S, +)$ for the semigroup $(\beta S, \Box)$, 
as in \cite{HS}, where we recall that, in general,  $x+y \neq y+x$ for $x,y \in \beta S$.

 There is also  a unique binary operation $\Diamond $  on $\beta S$ such that $ (\beta S, \Diamond)$
 is a semigroup containing $S$ as a subsemigroup and such that $(\beta S, \Diamond)$ is a compact, left topological semigroup. In the case where the semigroup 
$S$ is abelian,  the two semigroups  $(S^*,\Box)$ and $(S^*, \Diamond)$ have the same minimal ideal and $\ell^{\,1}(\beta S, \Diamond)$
 is just the opposite algebra to $\ell^{\,1}(\beta S, \Box)$, and so these two algebras have the same Jacobson radical.  In the case where $G$ is a group,
 the map $s\mapsto s^{-1}$ on $G$ extends to a continuous homeomorphism $\eta : \beta G \to \beta G$ such that 
$\eta(x \bb y) = \eta(y)\diamond \eta(x)\,\;(x,y \in \beta G)$. 
It follows easily that $(\beta G, \Box)$ is semisimple if and only if $(\beta G, \Diamond)$ is semisimple; we do not know if this is true when we replace
 $G$ by a (cancellative) semigroup. \s

Let $S$ be a semigroup.  We note that the map $L_s$  is continuous on  $ (\beta S, \Box)$ for each $s \in S$, and that, for many  semigroups $S$,
 including all weakly cancellative semigroups, the map  $L_s$ on $ \beta S $ is continuous only if $s \in S$ \cite[Theorem 12.20]{DLS1}.

Let $S$ be a semigroup,  and take  $u\in   \beta S$. Recall that the left ideal $(\beta S)u$ is closed in $\beta S$ and that $(\beta S)u = \overline{S u}$; 
we shall use this fact several times.
 The set  $S^*$ is an ideal in $\beta S$ if and only if $S$ is weakly cancellative \cite[Theorem 6.16(ii)]{DLS1}, and then 
  $S^*= (S^*, \Box)$ is also a compact, right topological semigroup; further,    $\ell^{\,1}(S^*)$ is a closed ideal in $(\ell^{\,1}(\beta S),\,\star\,)$, and hence 
$\ell^{\,1}(\beta S) = \ell^{\,1}(S)\ltimes \ell^{\,1}(S^*)$ as a semi-direct product. The   structure theorem  applies to both $\beta S $ and 
$S^*$; in particular, $\beta S$ and $S^*$  each  have a (unique) minimum ideal.  In the case where  $S$ is  weakly cancellative, $K(S^*) = K(\beta S)$.

\s
 
\begin{proposition} \label{3.2a}
Let $S$ be a weakly cancellative semigroup such that $\ell^{\,1}(S)$ is semi\-simple. 
 Then $J(\beta S) =J(S^*)$.
\end{proposition}

\begin{proof}  By Proposition \ref{1.2}(i), $J(\beta S) \cap \ell^{\,1}(S^*) = J(S^*)$. Since  $\ell^{\,1}(S)$ is semi\-simple,
   $J(\beta S)\subset \ell^{\,1}(S^*)$ by Proposition \ref{1.2}(iii).
\end{proof}\s

In particular, $J(\beta S) =J(S^*)$  whenever $S$ is either $\N$ or a group. \s
 
\begin{proposition} \label{3.2b}
The algebra  $\ell^{\,1}(\N^*)$ is semisimple if and only if $\ell^{\,1}(\Z^*)$ is semi\-simple.
\end{proposition}

\begin{proof} By \cite[Exercise 4.3.5]{HS}, $\N^*$ is a left ideal in $\Z^*$, and so  $\ell^{\,1}(\N^*)$ is a closed left ideal in $\ell^{\,1}(\Z^*)$. 
 By \cite[Theorem 8.34]{HS}, there is an element $x \in \N^*$ such that $x$ is cancellable in $\Z^*$, and so $u=\delta_x \in \ell^{\,1}(\N^*)$ has the property
 that    $au \neq 0$ and $ua\neq 0$ whenever $a \in \ell^{\,1}(\Z^*)\setminus \{0\}$.  Thus the result follows from Proposition \ref{1.2b}.
\end{proof}\s

\begin{example} For $m,n \in\N$, define $m\vee n =\max\{m,n\}$, and set $S = (\N,\vee)$.  Then $S$ is a countable, weakly cancellative, abelian semigroup, 
and  $\ell^{\,1}(S)$ is  semisimple because $S$ is separating (see also \cite[Example 4.9]{DLS1}), and so  $J(\beta S) =J(S^*)$.  Take $u,v \in S^*$,
 then $u\bb v= v$, and so $(S^*, \Box)$ is a right 
zero semigroup. It is easy to see \cite[Example 7.32]{DLS1}, that
 $$
J(\beta S, \Box) =\left\{f \in \ell^{\,1}(S^*): \sum_{u\in S^*} f(u)=0\right\}\,,
$$
and so $\ell^{\,1}(S^*)$ is not semisimple. \quad\qedhere
\end{example}\s

The following result shows immediately that $\{0\}  \subsetneq   {\mathcal N}(\N^*, +)   \subsetneq  {\mathcal Q}(\N^*, +)$.  
The theorem is due to Hindman and Pym \cite{HP}; see \cite[$\S7.3$]{HS} for more general results.\s
 
\begin{proposition}\label{4.4} 
The semigroup  $( \N^*, +)$ contains many isomorphic  copies  of $\,\F_2$ as a subgroup of $K(\N^*)$, and $\ell^{\,1}(\N^*, +)$
contains many   isometric and isomorphic copies of $\ell^{\,1}(\F_2)$ as a closed subalgebra.\quad\qedhere 
\end{proposition}\s

In this paper, we shall seek to determine the space $J(S^*)$ for a semigroup $S$, concentrating on the case where  $S$ is cancellative,  countable, and abelian, 
and more generally for a  countable semigroup $S$ that can be embedded in a group.  
We  shall see that it seems to be difficult to determine even whether $J(\N^*)$ is equal to $\{0\}$, and hence that $\ell^{\,1}(\N^*, +)$ is a semi\-simple Banach 
algebra: we shall show that this question is closely related  to  well-known open questions in the theory of $\beta \N$.  
For a general cancellative,  countable,   abelian semigroup $S$, we should like to determine $J(S^*)$  if it should transpire that $\ell^{\,1}(S^*)$ is not semisimple.
\medskip

\section{Results about $\beta S$}\noindent In this section, we establish some results about the Stone--\v{C}ech semigroup compact\-ification 
of a semigroup which will usually be cancellative and countable. 

We shall use a version of \cite[Theorem 3.40]{HS} several times; for convenience, we re-state this result here.\s

\begin{lemma}\label{6.0}  Let $S$ be a non-empty set,  and let $P$ and $Q$ be countable subsets of $\beta S$.  Suppose that 
$\overline{P}\cap \overline{Q} \not= \emptyset$.
Then either  $ {P}\cap \overline{Q} \not= \emptyset$ or $\overline{P}\cap {Q} \not= \emptyset$.\quad\qedhere 
\end{lemma}\s

We recall from \cite{DLS1} and \cite{HS} the definition of a specific subsemigroup ${\mathbb H}$ of $(\N^*, +)$.

Let $n \in \N$. Then   $\Z_n = \{0, 1, \dots,n-1\}$  is the natural group with respect to addition modulo
 $n$, so that there is a quotient map $q_n : \Z\to \Z_n$ which is a group homomorphism; the map $q_n$ extends to a semigroup homomorphism $q_n : \beta \Z \to \Z_n$.
  The subset  ${\mathbb H}$ of $\N^*$ is defined to be 
\[
{\mathbb H}= \{ x\in\beta \N : q_{2^n}(x) =0\,\;(n\in\N)\}\,.
\] 
We note that ${\mathbb H}$ is a subsemigroup of $(\N^*,+)$ and a   $G_\delta$-set in $\N^*$ and that ${\mathbb H}\supset E(\N^*)$, so that
  ${\mathbb H}\cap K(\N^*) \neq \emptyset$.\s

  We shall require the following notion and theorem from \cite{HS}. 
  
Let $G$ be a group, and suppose that there is a monomorphism  $\gamma : G\to C$, where $C$ is a compact topological group. We identify $G$ as
 a subset of $C$; we may suppose that $G$  is dense in $C$.  In the case where $G$ is countable, we may also suppose that  $C$ is metrizable. 
 We note that every abelian group can be embedded in a compact group which is a product of copies of the circle group; also,  $\F_2$ 
can be algebraically embedded in a compact topological group \cite[Proposition 2.24]{HS}.
There is an extension of $\gamma$ to a continuous epimorphism $\gamma : (\beta G, \Box) \to C$.  We define $V$ \label{V} to be the kernel of $\gamma$, and,
 for $x,y \in \beta G$, we set $x\sim y$ if $\gamma(x) = \gamma(y)$.
 
 The following theorem follows from \cite[Theorem 7.28]{HS}.\s
  
  \begin{theorem}\label{4.3}
Let $G$ be a countably infinite group, and let $V$ be as above.  Then $V$ contains
 $E(G^*)$, and $V$ is topologically isomorphic to ${\mathbb H}$.\quad\qedhere
\end{theorem}\s 
  
 Thus, in the above case, there is a map $\theta : \ell^{\,1}(V) \to   \ell^{\,1}({\mathbb H})$ such that $\theta$ is an isometry and an algebra 
isomorphism; in particular, $J(V)$ can be identified with $J({\mathbb H})$.\s

\begin{lemma}\label{6.7}
Let $G$ be a countably infinite group, and let $E$ be an equivalence class determined by the relation $\sim$.  Then there is a cancellable element
 $u\in\beta G$ such that $uE \subset  V$.
\end{lemma}

\begin{proof} We suppose that $G$ is embedded in  a compact topological group $C$, as above, and   take $c \in C$ such that $\gamma(x) = c$ for each $x \in E$. 
 Let $(U_n)$ be  a   sequence which is a  basis for the  family  of open neighbourhoods of the identity of $C$.  For each $n\in\N$, choose $s_n \in G$
 such that $s_nc \in U_n$, and set $S =\{s_n : n\in\N\}^*$,  so that $S$ is a clopen subset of $G^*$.  Clearly, for each $s\in S$, we have $sE \subset V$.
 By \cite[Theorem 8.34]{HS}, $S$  contains a cancellable element of $\beta G$. \end{proof}\s

Our  first results are modifications of Theorems 6.56 and 6.57 of \cite{HS}.      
 We adopt the following notation, which we shall maintain  throughout this section.   Let $S$ be a   countable  semigroup  that is a subsemigroup of a  group $G$; 
we may suppose that  $G$ is also countable.   For example, starting from a countable, cancellative, abelian  semigroup $S$, we can take $G$ to be 
the group of quotients of $S$.  We order  the group $G$ by a total ordering,  which we call `$<$'.

Let $(x_n : n\in\N)$  be a sequence in $S^*\setminus K(S^*)$, and fix $q \in K(S^*)$.  Then clearly we have 
$(\beta G)\,qx_n= \overline{Gqx_n} \subset K(S^*)\,\;(n\in \N)$, and so 
\[
\{x_1,\dots, x_n\} \cap \bigcup_{i=1}^n (\beta G)\,qx_i=\emptyset\quad (n\in\N)\,.
\]
Thus, for each $n\in\N$, we can choose a clopen subset $W_n$ in $\beta S$ such that
\begin{equation}
\{x_1,\dots, x_n\}\subset W_n\quad{\rm and}\quad W_n \cap \bigcup_{i=1}^n (\beta G)\,qx_i=\emptyset\,.
\end{equation}
For each $n\in\N$ and $r \in G$, we set
\begin{equation}
 U_{n,r} = \{u\in S^* : rux_i \not\in W_n\,\;(i=1,\dots,n)\}\,.
\end{equation}
We note that each $ U_{n,r}$ is a clopen subset of $S^*$   containing $q$.  
 
Since $G$ is countable, the intersection $\bigcap\{ U_{n,r}: n\in\N,\,r \in G\}$ is a $G_\delta$-set in $S^*$, and so it has a non-empty interior \cite[Theorem 3.36]{HS}.
 Thus we can find and fix a non-empty, clopen subset $U$ of  $S^*$ such that  $U \subset U_{n,r}$  for each $n\in\N$ and $r \in G$.
 By \cite[Theorem 8.34]{HS}, the set of cancellable elements of $\beta S$
 contains a dense, open subset of $S^*$, and so, by intersecting $U$ with such a set, we may suppose that every element of $U$ is cancellable in $\beta S$.

 In the special  case in which $S=\N$ and $G=\Z$,  we can suppose that we have chosen $q\in  {\mathbb H}$ (because ${\mathbb H}\cap K(\N^*) \neq \emptyset$) and that 
 $U \subset  {\mathbb H}$.  This follows from the fact that ${\mathbb H} \cap \bigcap\{ U_{n,r}: n\in\N,\,r \in \Z\}$ is a non-empty, 
$G_\delta$-set in $ \N^*$.
 
The set $U$ has the form $A^*$ for some infinite subset $A$ of $\N$; we write $A= \{a_1,a_2, \dots\}$.
By passing to a subset of  $A$, if necessary, we may suppose that
\begin{equation}\label{(6.3)}
ba_m \neq a_n\quad {\rm whenever}\quad m<n\;\;{\rm in}\; \;\N\quad {\rm and} \quad b<a_m \;\;{\rm in}\;\; G\,.
\end{equation}
For each $r\in G$, we set  $A_r = \{a  \in A : r< a\}$. Of course, $A \setminus A_r$ is finite, and so $A_r^* =A^*\,\;(r\in G)$.  \s

\begin{lemma}\label{6.1}
For each $u \in U$ and $m,n\in\N$,  we have $x_m\not\in (\beta G)\,ux_n$. 
\end{lemma}

\begin{proof}  Take $u \in U$, and assume towards a contradiction that there exist $m,n\in\N$ such that $x_m\in (\beta G)\,ux_n= \overline{Gux_n}$.
Take $k \in \N$ with $k> \max\{m,n\}$.  Then $W_k$ is an open neighbourhood of $x_m$, and so there exists $y\in G$ such that $yux_n \in W_k$.  But this 
contradicts the fact that $u \in U_{k,y}$.  Thus $x_m\not\in (\beta G)\,ux_n$.
\end{proof}\s

\begin{lemma}\label{6.2}
For each  $u \in U$ and   $ n\in\N$,   the element $ux_n$ is right cancellable in $\beta G$.
\end{lemma}

\begin{proof}  Assume towards a contradiction that $ux_n$ is not right cancellable in $\beta G$. 
 By \cite[Theorem 8.18, (3) $\Rightarrow$ (1)\/]{HS}, there exists 
$x \in  G^*$ such that $ux_n= xux_n$.  Since $ux_n\in \overline{Ax_n}$  and $xux_n \in \overline{Gux_n}$, it follows from Lemma \ref{6.0} 
that one of the following two alternatives must hold:\s

(i) $vx_n = rux_n$ for some $v\in \overline{A}$ and some $r\in G\,$;\s

(ii) $ax_n = yux_n$ for some $a\in A$ and some $y\in \beta G\,$.\s

Suppose that (i) occurs.    Assume that $v \in S$. Then $v^{-1}rux_n= x_n \in W_n$, a contradiction of the fact that 
$u \in U_{n,v^{-1}r}$.  Thus  $v \in A^*$.   It follows that $vx_n  \in \overline{Ax_n} = \overline{A_rx_n}$.  Also,
 $rux_n\in \overline{rA_rx_n}$. By a second application of Lemma \ref{6.0}, one of the following two alternatives must hold:\s

(iii)  $bx_n =ru_1x_n$ for some $u_1\in \overline{A_r}$ and some $b\in A_r\,$;\s

(iv) $u_2x_n =rcx_n$  for some $u_2\in \overline{A_r}$ and some $c\in A_r\,$.\s

Now case (iii)  cannot hold when $u_1 \in A_r$  by (\ref{(6.3)}), and case  (iii)  cannot hold when $u_1 \in A_r^*$ because, 
in this case,  $u_1 \in  U_{n,b^{-1}r}$ and so $x_n = b^{-1}ru_2x_n \not\in W_n$, a contradiction of the fact that  $x_n \in W_n$.  Thus (iii) cannot hold. 
 Similarly, (iv) cannot hold.
 
 We have obtained  a contradiction in the case where (i) holds.
 
Now suppose that (ii) occurs.   Since $a^{-1}yux_n \in \overline{Gux_n}$ and $a^{-1}yux_n = x_n\in W_n$, it follows that there exists $t\in G$ such that 
 $tux_n \in W_n$, a contradiction of the fact $u\in U_{n,t}$.  Thus we have obtained  a contradiction also in the case where (ii) holds.
\end{proof}\s

\begin{lemma}\label{6.3}
For each  $u \in U$ and each $m,n\in\N$,  either $x_m \in Gx_n$ or
 \[
(\beta G)\,ux_m\cap (\beta G)\,ux_n=\emptyset\,.
\]
\end{lemma}

\begin{proof} Take $k\in\N$ with $k> \max\{m,n\}$.

Suppose that $(\beta G)\,ux_m\cap (\beta G)\,ux_n\neq \emptyset$.  By \cite[Corollary 6.20]{HS}, we may suppose that
 $xux_m =ux_n$ for some $x\in \beta G$.  Now $xux_m \in \overline{Gux_m}$ and $ux_n \in \overline{Ax_n}$, and so it again follows 
 from Lemma \ref{6.0}  that one of the following two alternatives must hold:\s

(i) $sux_m =vx_n$ for some $v\in \overline{A}$ and some $s\in G\,$;\s

(ii) $yux_m =ax_n$  for some $a\in  A$ and some $y\in \beta G\,$.\s

Assume towards a contradiction that (i) holds.  Again we see that  $sux_m \in \overline{sAx_m}$ and  $vx_n \in \overline{Ax_n}$,
 and so it again follows   from Lemma \ref{6.0}  that one of the following two alternatives must hold:\s
 
 (iii) $su_1x_m = bx_n$ for some $u_1\in \overline{A}$ and some $b \in A\,$;\s
 
 (iv) $scx_m = u_2x_n$ for some $u_2\in \overline{A}$ and some $c \in A\,$.\s
 
 However (iii) cannot hold in the case where $u_1\in A^*$ because this would contradict the fact that $u_1 \in U_{k,b^{-1}s}$. Hence $u_1 \in A$, and so 
we can conclude that $x_m \in Gx_n$.  Similarly, (iii) cannot hold in the case where $u_2\in A^*$, and so again  $x_m \in Gx_n$.
\end{proof}\s

 We  set  $x\equiv y$ for $x,y \in  \beta G$ if $x\in Gy$. 

It follows from the above lemmas that we have the following theorem. \s

 \begin{theorem}\label{6.4}  Let $S$ be a   countable  semigroup  that is a subsemigroup of a  group $G$, and
suppose that  $(x_n : n\in\N)$  is a sequence in $S^*\setminus K(S^*)$. Then there is an infinite subset $A$ of $S$ such that, for   each  $u\in A^*$, the 
following properties hold:\s
 
 {\rm (i)} $u$ is cancellable;\s
 
 {\rm (ii)}  $ux_n$ is right cancellable for each $n\in\N$;\s
 
  {\rm (iii)}  for  each $m,n \in\N$, either $x_m \equiv x_n$ or $(\beta G)\,ux_m\cap (\beta G)ux_n=\emptyset$.\quad\qedhere  
 \end{theorem}\medskip
 
 \section{The radical of some semigroup algebras} \noindent Here we begin to study $J(S^*)$, the radical of $\ell^{\,1}(S^*)$, for suitable semigroups $S$.
In part\-icular, for the remainder of the paper, our semigroups are abelian.\s

 \begin{theorem}\label{6.5}
Let   $S$ be a cancellative, countable, abelian semigroup, and  suppose that  $f\in J(S^*)$ or $f\in J_0(S^*)$.   Then $\supp f \subset  K(S^*)$.   
\end{theorem}

\begin{proof} We take $G$ to be the group of quotients of $S$, so that $S$ is a subsemigroup of $G$ and Theorem \ref{6.4} applies. 
 We now denote the semigroup operation in $G^*$ by `$+$'; for $x \in \beta G$ and $n\in\N$,  we write $n\,\ast\,x$ for
 $x  + \cdots+ x$, where there are $n$ copies of $x$.

Assume towards a contradiction that $\supp f \not\subset  K(S^*)$, and  set  
\[
 X= \supp f \setminus K(S^*)\,,
\]
so that $X$ is a countable, non-empty set.

By Theorem \ref{6.4},  there exists $u\in \beta S$ such that $u$ is cancellable, such that $ux$ is right cancellable for each $x\in X$, and, furthermore, for each
$x,y\in X$, either $x \equiv y$ or $(\beta G)\,ux\cap (\beta G)\,uy=\emptyset$.
 By replacing  each $x\in X$ by $ux$ and 
 replacing $f$ by $\delta_u\,\star\, f$, we may suppose that $x$ is right cancellable for each $x\in X$ and   that, for each
$x,y\in X$, either $x \equiv y$ or $(\beta G)\,x\cap (\beta G)\,y=\emptyset$.
Note that it remains true that $f \in J(S^*)$ or $f\in J_0(S^*)$ because $J(S^*)$ and  $J_0(S^*)$ are  ideals in $\ell^{\,1}(S^*)$ and $\C S^*$, respectively, 
and so, in either case, $\lim_{n\to\infty}\lV f^{*n}\rV^{1/n} =0$.  Further,  $\lV f\rV =\lV \delta_u\bb f\rV $ because
 $u$ is cancellable, and so we have not changed the value of $\lV f\rV$.
 
 Suppose that 
\begin{equation}\label{(6.5)}
x_{i_1}+ \cdots +x_{i_k}\equiv x_{j_1}+\cdots +x_{j_m}\,,
\end{equation}
 where $x_{i_1}, \dots,x_{i_k}, x_{j_1},\dots,x_{j_m}\in X$.  Then $(\beta G  + x_{i_k})\cap (\beta G + x_{j_m})\neq \emptyset$, and so $x_{i_k}\equiv x_{j_m}$. 
Since $x_{i_k}$ and $x_{j_m}$ are right cancellable, it follows that 
\[
x_{i_1}+ \cdots +x_{i_{k-1}}\equiv  x_{j_1}+\cdots +x_{j_{m-1} }\,.
\]
By repeating this argument, we see that it follows from (\ref{(6.5)}) that $k=m$ and $x_{i_r}\equiv x_{j_r}$ for all $r\in\{1,...,k\}$. 
 
Choose $x \in X$, and set $T_n = G + n\,\ast\,x$ for $n\in \N$.  Set $h = f\mid T_1$, so that $h\in \ell^{\,1}(S^*)$. Since $f(x)\neq 0$, we have  $h(x)\neq 0$.  
By the remark of the previous paragraph, it follows that, for each $n \in \N$, we have
$h^{*\,n} = f^{*\,n}\mid T_n$, and so $\lV  h^{*\,n} \rV \leq  \lV  f^{*\,n} \rV$.  Consequently, $\lim_{n\to\infty}\lV h^{*n}\rV^{1/n} =0$ and $h \in {\mathcal Q}(S^*)$.  
  Now define $\varphi \in \ell^{\,1}(G)$ by 
 \[
\varphi (y ) = h(y+x)\quad (y\in G)\,.
\]
Then $\lV \varphi^{*\,n}\rV \leq \lV h^{*\,n}\rV \,\;(n\in\N)$, and so $\varphi \in  {\mathcal Q}(G)$.  However $ {\mathcal Q}(G)=\{0\}$
 because $G$ is an abelian group, and so $\varphi =0$. Hence $h(x) =0$, a contradiction.
 
 We conclude that  $\supp f \subset  K(S^*)$.
\end{proof}\s

\begin{corollary}\label{6.5a}
Let   $S$ be a cancellative, countable, abelian semigroup.  Then  $\ell^{\,1}(S^*)$  is semisimple if and only if  $\ell^{\,1}(K(S^*))$ is semi\-simple. 
\end{corollary}\s

\begin{proof} Assume that $\ell^{\,1}(K(S^*))$ is semisimple, and take $f \in J(S^*)$. Then, by the theorem, $\supp f \subset  K(S^*)$,  and so 
$f\in J(S^*)\cap \ell^{\,1}(K(S^*)) \subset J(K(S^*))=\{0\}$. Thus  $f=0$, and so  $\ell^{\,1}(S^*)$ is semi\-simple.

Assume that $\ell^{\,1}(S^*)$ is semi\-simple. By Proposition \ref{1.2}(i), $J(K(S^*)) =\{0\}$, and so $\ell^{\,1}(K(S^*))$ is semisimple.
\end{proof}\s

\begin{proposition}\label{6.6}
Let $f\in J({\mathbb H})$. Then $\supp f \subseteq K\cap {\mathbb H}$.
\end{proposition}

\begin{proof}   We observed in the course of the above discussion that, in the case where $S=\N$ and $G=\Z$, we could have chosen our non-empty subset
 $U$ to be a subset of ${\mathbb H}$. Then the given proof leads to  the stated result.
\end{proof}\s

A {\it rectangular semigroup\/} is a semigroup $R$ that, as a set, has the form $A \times B$, where $A$ and $B$ are non-empty sets, and the product is given by
$(a,b)(c,d) = (a,d)$ for $a,c \in A$ and $b,d \in B$, so that all elements of $R$ are idempotents. Let $R =A\times B$ be such a semigroup.
 In the following we denote the semigroup action by  juxtaposition, and we write $\pi_A$ and $\pi_B$  for the projections onto $A$ and $B$, respectively.
 Fix two distinct  elements,  $b_1$ and $b_2$ in $B$,  and consider  the set $U$ of  pairs $\{u,v\}$ of elements $R$ such that $\pi_B(u) = b_1$ and 
 $\pi_B(v) = b_2$.  Note that, for $\{u_1,v_1\}$  and $\{u_2,v_2\}$  in $U$, we have
\begin{equation}\label{(6.6)}
u_1u_2 = u_1,\quad u_1v_2 = u_2,\quad v_1u_2=u_1,\quad v_1v_2 =v_1\,.
\end{equation}
Also note that the set $U$ is closed under left-translation by elements of $R$.

Consider the set $N$ of elements $f\in \ell^{\,1}(R)$ of the form  $f = \delta_u-\delta_v$, where $\{u,v\}\in U$, so that $N \subset \C R$.
 Then it follows from (\ref{(6.6)}) that $f_1\,\star\,f_2 = 0$ whenever $f_1,f_2 \in N$.   Further, $N$ is closed under left-translations by elements
of $R$.   Take $f \in N$ and $g\in \ell^{\,1}(R)^\sharp$. Then $g\,\star\,f$ has the form $h= \sum_{i=1}^\infty \alpha_i f_i$, where 
$\alpha_i\in \C\,\;(i\in\N)$, $\sum_{i=1}^\infty \lv \alpha_i \rv < \infty$, and $f_i\in N\,\;(i\in\N)$.  Thus  $h\,\star\,h = 0$.   We conclude that each such element 
$g\,\star\,f$  is nilpotent of index at most 2, and so $f\in J(R)\cap J_0(R)$. Thus $N\subset J(R)\cap J_0(R)$.    

This implies the following result.\s

\begin{proposition}\label{6.8}
Let $R= A\times B$ be a rectangular semigroup with $\lv B \rv \geq 2$.  Then $\dim J(R) \geq \lv A \rv$ and $\dim J_0(R) \geq \lv A \rv$. In particular,
the algebras  $\ell^{\,1}(R)$  and $\C R$ are not semisimple.\quad\qedhere
\end{proposition}

The result has relevance to our main question because it is a result of the third author (\cite{Ze}, \cite[Theorem 9.41]{HS}) that $\N^*$ contains a copy of such a 
rectangular semigroup $R = A\times B$  with $\lv A \rv = \lv B \rv = 2^{\,\mathfrak c}$. 
 Thus we have many `very large' semigroups $R$ in $\N^*$ such that $\ell^{\,1}(R)$  and $\C R$ are far from being  semisimple.
\medskip

\section{A condition for semisimplicity} 

\noindent We now give our main description of $J(\beta S)$ and $J(S^*)$ for a cancellative, countable, abelian semigroup $S$.\s

\begin{theorem}\label{5.0}
 Let $W$ be a compact, right topological semigroup, and suppose that    $f\in \ell^{\,1}(K(W))$. 
Then the following are equivalent:\s

{\rm (a)}  $\delta_p\,\star\,f\,\star\, \delta_q =0$ for each $p,q \in K(W)$;\s

{\rm (b)}  $g_1\,\star\,f\,\star\,g_2\,\star\,f\,\star\,g_3\,\star\,f= 0$ for each $g_1,g_2,g_3 \in \ell^{\,1}(W)$;\s

{\rm (c)}  $f \in J(W)$.
\end{theorem}

\begin{proof}    We suppose that $f =\sum_{i=1}^\infty \alpha_i\delta_{x_i}$, where $\alpha_i\in \C$ 
and $x_i\in K(W)$ for $i\in\N$ and where  $\sum_{i=1}^\infty\lv \alpha_i\rv < \infty$. \s

(a) $\Rightarrow$ (b)  It suffices to prove (b) in the special case in which $g_1 = \delta_{y_1}$,  $g_2 = \delta_{y_2}$, and  $g_3 = \delta_{y_3}$ 
for some $y_1,y_2,y_3 \in W$. But in this case
\[
(g_1\,\star\,f)\,\star\,(g_2\,\star\,f)\,\star\,(g_3\,\star\,f)= \delta_{y_1}\,\star\,\sum_{i,j=1}^\infty {\alpha_i\alpha_j }(\delta_{x_i y_2}\,\star\,f\,\star\,
\delta_{y_3x_j})\,.
\]
Since $x_iy_2, y_3x_j \in K(W)$ for each $i,j \in\N$, it follows from (a) that each term in the bracket is $0$, 
 and so $g_1\,\star\,f\,\star\,g_2\,\star\,f\,\star\,g_3\,\star\,f= 0$.\s
 
 (b) $\Rightarrow$ (c)  By (b), $g\,\star\, f$ is nilpotent of index at most  $3$ for each $g\in \ell^{\,1}(W)$. More directly, $f$ itself
  is nilpotent of index at most  $3$. Thus  (c) follows  from Theorem \ref{1.1}.\s
 
(c) $\Rightarrow$ (a) Take $p,q \in K(W)$, and set $G= pK(W)q$, so that $G$ is a subgroup of $W$ by Theorem \ref{str}(iii).  Since  
$\supp(\delta_p\,\star\,f\,\star\, \delta_q) \subset G$, it follows from Proposition \ref{2.20} that $\delta_p\,\star\,f\,\star\, \delta_q=0$,
giving (a).\end{proof}\s

Suppose that $f\in \C K(W)$, in  the above notation.  Then the theorem still holds, with clause (c) replaced by `$f \in J_0(K(W))$'; in the proof of the implication 
(c) $\Rightarrow$ (a), we use  Proposition \ref{2.21}, rather than Proposition \ref{2.20}.    It follows that  $J(W) \cap \C W = J_0(W)$. \s

\begin{theorem}\label{5.1b}
Let $S$ be a cancellative, countable, abelian semigroup, and 
suppose that  $f \in \ell^{\,1}(S^*)$. Then $f\in J(S^*)$ if and only if $\supp f \subset K(S^*)$ and $\delta_p\,\star\,f\,\star\, \delta_q =0$ for each $p,q \in K(S^*)$.

Further, in this case, $g\,\star\,f$ is nilpotent of index at most 3 for each $g \in \ell^{\,1}(S^*)$.
\end{theorem}

\begin{proof}   Suppose that $f\in J(S^*)$. Then  $\supp f \subset K(S^*)$ by Theorem \ref{6.5}, and hence  $f \in \ell^{\,1}(K(S^*))$.  Now take $p,q \in K(S^*)$.  Then 
we have  $\delta_p\,\star\,f\,\star\, \delta_q =0$ by  the implication (c) $\Rightarrow$ (a) of Theorem \ref{5.0} (applied with $W=\beta S$).

Conversely, suppose that $f$ satisfies the two stated conditions.  Then $f \in \ell^{\,1}(K(S^*))$, and so $f\in J(S^*)$ by the implication (a) $\Rightarrow$ (c) 
of Theorem \ref{5.0}.

Now suppose that $f\in J(S^*)$. Then $(g\,\star\,f)^{*\,3}= 0$ for each $g \in \ell^{\,1}(S^*)$ by the implication (c) $\Rightarrow$ (b) of Theorem \ref{5.0}.
\end{proof}\s

Similarly, for an element $f \in \C S^*$, we have $f\in J_0(S^*)$ if and only if $\supp f \subset K(S^*)$ and $\delta_p\,\star\,f\,\star\, \delta_q =0$ 
for each $p,q \in K(S^*)$.

The above theorem concerns the algebra $(\ell^{\,1}(S^*),\Box)$. However, our earlier remarks show that the same characterization applies to the
 radical of  $\ell^{\,1}(S^*, \Diamond)$.  

We further remark that, for   each $f\in J(S^*)$, there exists $p\in K(S^*)$ such that $g\,\star\,f\,\star \,\delta_p$ is nilpotent of index at most 2 
for each $g \in \ell^{\,1}(S^*)$. Indeed, suppose that $f\,\star\,\delta_p  = 0$ for each $p\in K(S^*)$. Then this is immediate. Otherwise,
  $f\,\star\,\delta_p  \neq 0$  for some $p\in K(S^*)$, and then $\delta_q\,\star\,f\,\star\,\delta_p =0$ for each $q\in K(S^*)$, again giving the result.
 \s
  
\begin{theorem}\label{6.8a}
 The following  statements are equivalent:\s

{\rm (a)} for some infinite, countable, abelian group $G$, the algebra $\ell^{\,1}(G^*)$ is semisimple;\s

{\rm (b)}   for each  infinite, countable, abelian group $G$, the algebra $\ell^{\,1}(G^*)$ is semisimple;\s

{\rm (c)}  $\ell^{\,1}({\mathbb H})$ is semisimple;\s

{\rm (d)}  $\ell^{\,1}(\N^*)$ is semisimple.  
\end{theorem}

\begin{proof}  We consider the subset $V$ of $\beta G$ that was defined on page \pageref{V}.  We  note that
 $K(\N^*)\cap {\mathbb H} = K({\mathbb H})$, which is topologically isomorphic to the ideal $K(V)= V \cap K(\beta G)$ (see \cite[Theorem 1.65]{HS}).
Thus, the tequivalence of (a), (b), and (c)  will   follow once we have shown that, for a fixed  infinite, countable, abelian group $G$, 
the algebra $\ell^{\,1}(G^*)$ is semisimple  if and only if $\ell^{\,1}(V)$ is semisimple.

First, assume that there exists $f \in  J(G^*)$ with $f\neq 0$. Then  $\supp f \subset K(G^*)$ by Theorem \ref{6.5},  and so we may suppose that  
\[
f=\sum_{i=1}^\infty\alpha_i\delta_{x_i}\,,
\]
 where $\{x_i: i\in\N\}\subset K(G^*)$. By Theorem  \ref{5.1b}, $\delta_p\,\star\,f\,\star\, \delta_q =0$ for each $p,q \in K(G^*)$.   
  We partition the set $\{x_i: i\in\N\}$ into equivalence classes with respect
 to $\sim$, say into the disjoint subsets $\{E_m: m\in\N\}$, and set $f_m= f\mid E_m$ for $m\in\N$. Since $f\neq 0$, there exists $m_0\in \N$ with $f_{m_0}\neq 0$.
Now suppose that $m_1,m_2 \in\N$ with $m_1\neq m_2$. For each $x,y\in \beta G$ with 
$p+ x +q = p+y+q$ for some $p,q \in K(G^*)$, necessarily $x\sim y$, and so  the elements $\delta_p\,\star\,f_{m_1}\,\star\, \delta_q  $ and 
$\delta_p\,\star\,f_{m_2}\,\star\, \delta_q $ have disjoint support for each $p,q \in K(G^*)$.  Hence $\delta_p\,\star\,f_{m}\,\star\, \delta_q =0$ for each $m\in\N$ and 
each $p,q \in K(G^*)$.  Since $\supp f_m \subset K(G^*)$, Theorem \ref{5.0} applies to show that  $f_m \in J(G^*)$ for each $m\in\N$; in particular, $f_{m_0} \in J(G^*)$.

By Lemma \ref{6.7}, there is a cancellable element $u\in\beta G$ such that $u+x \in V$ for each $x \in E_{m_0}$.  Thus $\delta_u\,\star\,f_{m_0}\neq 0$ and 
$\delta_u\,\star\,f_{m_0}\in J(V)$.   This shows that $J(V)\neq \{0\}$.\s

Second,  assume that there exists $f \in  J(V)$ with $f\neq 0$.  Then we have $\supp f \subset K(V)$ by Proposition \ref{6.6}; 
in particular, $f\in \ell^{\,1}(K(V))$, and so, again, Theorem \ref{5.0} applies.

 Take  $r \in E(K(\beta G))\subset V$. Since  $r+ K(V) +r$ is a subgroup of $V$, we have  $\delta_r \,\star f\,\star\,\delta_r =  0$.
Now take $p,q \in K(\beta G)$. By Theorem 2.1(iii), there exists  $r\in E(K)$  with   $p+ r =p\,$ and  $r+q=q$, and so 
\[
\delta_p\,\star\, f\,\star\,\delta_q = \delta_p\,\star\,\delta_r\,\star\,f\,\star\, \delta_r\,\star\,\delta_q =0\,.
\]
By Theorem \ref{5.0}, (a) $\Rightarrow$ (c), $f \in  J(G^*)$. This shows that $J(G^*)\neq\{0\}$.

By Proposition \ref{3.2b},  $\ell^{\,1}(\N^*)$ is semisimple if and only if $\ell^{\,1}(\Z^*)$ is semi\-simple, and so (d) is also equivalent to the other statements.
\end{proof}\s
 
 \begin{theorem}  Let   $S$ be a cancellative, countable, abelian semigroup. Consider the following conditions on $(S^*, +)$:\s
 
{\rm (a)}    there exist $n\in \N$ and two disjoint sets $\{x_1,\dots,x_n\}$ and $\{y_1,\dots,y_n\}$ of $K(S^*)$  such that, for each $p,q \in K(S^*)$, the set 
$\{p+x_1+q,\dots,p+x_n+q\}$  is a permutation of the set $\{p+y_1+q,\dots,p+y_n+q\}\,$; \s

{\rm (b)}  $J(S^*,+) \neq \{0\}$ and/or $J_0(S^*,+) \neq \{0\}$;\s

{\rm (c)}  there is a non-empty, finite subset $F$ of distinct elements of $K(S^*)$ and $x \in F$ such that, for  each $p,q \in K(S^*)$, there exist
  $y\in F$ with  $y\neq x$  and $p+y+q =p+x+ q$.\s

\noindent Then {\rm (a)}  $\Rightarrow$ {\rm (b)}  $\Rightarrow$ {\rm (c)}.
\end{theorem}

 \begin{proof} 
{\rm (a)} $\Rightarrow$ {\rm (b)}   Let $n\in \N$ and   $\{x_1,\dots, x_n\}$ and $\{y_1,\dots, y_n\}$ be as specified in (a), and  set 
\[
f=  \sum_{i=1}^n \delta_{x_i} - \sum_{i=1}^n\delta_{y_i}\,,
\]
 so that $f \in \C S$ with $\supp f \subset K(S^*)$ and $f \neq 0$.  Take $p,q \in K(S^*)$.  Then  clearly $\delta_p\,\star\,f\,\star\, \delta_q =  0$, and so, 
by Theorem \ref{5.1b}, $f \in J_0(S^*)\cap J(S^*)$. Hence $J(S^*,+) \neq \{0\}$ and  $J_0(S^*,+) \neq \{0\}$.\s
 
 {\rm (b)}  $\Rightarrow$ {\rm (c)} Take $f\in J(S^*)$ with $f\neq 0$; we may suppose that $\lV f \rV=1$.  By  Theorem \ref{6.5}, $\supp f \subset K(S^*)$,
 and so $f$ has the form $\sum_{i=1}^\infty \alpha_i\delta_{x_i}$, where $\alpha_i\in\C\setminus\{0\}\,\;(i\in\N)$, $\sum_{i=1}^\infty \lv \alpha_i\rv =1$, 
and $\{x_i: i\in\N\}$ is a set of distinct points in $K(S^*)$.   Choose $k\in\N$ such that $\sum_{i=k+1}^\infty \lv \alpha_i\rv < \lv \alpha_1\rv $,
 and set $F=  \{x_1,\dots,x_k\}$,  so that  $F$ is a non-empty, finite subset  of distinct elements of $K(S^*)$.  Set $g= \sum_{i=1}^k \alpha_i\delta_{x_i}$ and
 $h= \sum_{i=k+1}^\infty \alpha_i\delta_{x_i}$, so that $g, h\in \ell^{\,1}(\N^*)$, $f=g+h$,  $\lV g\rV \geq \lv \alpha_1\rv$, and $\lV h\rV < \lv \alpha_1\rv$.  
Set $F= \{x_1,\dots,x_k\}$ and $x = x_1$.

By Theorem \ref{5.1b}, $  \delta_p\,\star\,f \,\star\,\delta_q=0$ for each $p,q \in K(S^*)$.  Take $p,q \in K(S^*)$, and assume  towards a contradiction that, 
for each $ y \in F$ with $y\neq x$, we have $p+x+q \neq p+y+ q$.   Then $\lV  \delta_p\,\star\,g \,\star\,\delta_p \rV= \lV g\rV\geq \lv \alpha_1\rv$ 
and $\lV  \delta_p\,\star\,h \,\star\,\delta_p \rV < \lv \alpha_1\rv $,  and so $\lV  \delta_p\,\star\,f \,\star\,\delta_p \rV  >0$, 
a contradiction of the fact that $  \delta_p\,\star\,f \,\star\,\delta_q=0$. Thus  there exist $ y \in F$ with $y\neq x$ such that  $p+y+q = p+x+ q$. 
 \end{proof}\s
 
The question whether or not clause (a) of the above theorem holds is a well-known open question in the theory of Stone--\v{C}ech semigroup compactifications;
 in particular, it is open for the case  where $S = (\N,+)$.   Indeed, it may be that there exist  $x,y\in K(S^*)$ with 
 $x\neq y$ and such that $p+x+y =p+y+ q$ for each $p,q \in K(S^*)$, a condition that implies (a).  Unfortunately, we do not know whether the conditions in 
clauses (a) and (c) are equivalent.
 \medskip
 
 \section{Measure algebras}
 
\noindent Let $S$  be a semigroup.  Then  $M(\beta S)$ denotes  the Banach space of complex, regular Borel measures  on the compact space $\beta S$, 
with the total variation norm.  There are two Arens products, $\Box$  and $\,\Diamond\,$,  on $M(\beta S)$;  they are defined by  identifying 
$M (\beta S)$  with $\ell^{\,1}(S)''$.  Full details of this identification are given  in \cite[Chapter 7]{DLS1}.    
 The restriction of the products  $\Box$ and $\Diamond$ to elements $s \in \beta S$ (when $s$ is identified with the point mass $\delta_s \in M(\beta S)$) 
 coincides with the previous definitions of $\Box$ and $\Diamond$ on $\beta S$. We shall consider  $M (\beta S)$ to be a Banach algebra with respect to the product 
$\Box$.
 
Set
\[
J_1 = \{\mu \in M(\beta S): \delta_s\,\Box\, \mu =\mu\,\;(s\in S),\;\,\mu(\beta S) = 0\}\,.
\]
It is easily seen that $J_1$ is a closed, nilpotent  ideal of index $2$ in $M(\beta S)$, and hence $J_1 \subset J(M(\beta S))$. 
  Thus  $M(\beta S)$ is not semisimple whenever $J_1 \neq \{0\}$.
  
  In \cite[Proposition 7.21]{DLS1}, it is shown that $J_1$ is infinite dimensional for many semigroups $S$, including the case where $S$ 
is an amenable group; in  fact, the dimension of $J_1$ is `large'  \cite[Theorem (7.3)(ii)(b)\,]{Pat} in this later case, and so  the dimension of
$J(M(\beta S))$ is also large. See also  \cite[Theorem 7.22]{DLS1},
 where a somewhat larger ideal than $J_1$ -- say it is $J_2$ -- is constructed, and it is shown that
 $J_2 \subset  J(M(\beta S))$.     The original result that $(M(\beta G), \Box)$ is not semisimple whenever $G$ is an amenable group is due to Granirer \cite{Gran}. 
 
 This leaves open the question of a description of the radical of  $M (\beta G)$  whenever $G$ is a non-amenable group, such as $\F_2$. 
 It is a conjecture  that $M(\beta\, \F_2)$ is semisimple.  The following is a partial remark towards this; the hypotheses on $G$ in the following theorem
 are satisfied by  $\F_2$.\s
 
 \begin{proposition}\label{6.10}
Let $G$ be a countable group that can be embedded in a compact topological group.
 Assume that $\ell^{\,1}(\N^*,+)$ is not semisimple. Then $(M(\beta G), \Box)$ is not semisimple.
 \end{proposition}
 
 \begin{proof} Take  $V$ to be the subset of $G$ defined on  page \pageref{V}. Since $\ell^{\,1}(\N^*)$ is not semisimple, 
it follows from Proposition \ref{6.8a}  that  there exists $f\in J(V)$ with $f\neq 0$ and such that 
$\supp f \subset K(V) = K(G^*) \cap V \subset K(G^*)$.  Thus we can apply Theorem \ref{6.5} to see that 
\[
g_1\,\star\,f\,\star\,g_2\,\star\,f\,\star\,g_3\,\star\,f= 0\quad (g_1,g_2,g_3 \in \ell^{\,1}(\beta G))\,.
\]
 
 Now take ${\rm M} \in M(\beta G)$.  First replace $g_3$ by a net in $\ell^{\,1}(\beta G)$ that converges to ${\rm M} $ in the weak-$*$ topology.  Then 
\[
g_1\,\star\,f\,\star\,g_2\,\star\,f\,\Box\,{\rm M}\,\Box\,f= 0\quad (g_1,g_2  \in \ell^{\,1}(\beta\, \F_2))\,.
\]
 Similarly we see successively that 
 ${\rm M}\,\Box\,f\,\Box\,{\rm M}\,\Box\,f\,\Box\,{\rm M}\,\Box\,f= 0$. Thus  ${\rm M}\,\Box\,f$ is nilpotent of index at most $3$ in $M(\beta G)$ for
 each ${\rm M} \in M(\beta G)$, and so, by Theorem \ref{1.1}, $f \in J(M(\beta G))$.  This shows that $M(\beta G) $ is not semisimple.
 \end{proof}\medskip

\section*{Acknowledgements}   This research was commenced when the first and second authors were invited to visit the University of Witwatersrand in December, 2011. 
They are very grateful for the generous hospitality received.

\newcommand{\email}{\texttt}

\noindent   H.\ Garth\ Dales, \\
 Department of  Mathematics and Statistics\\
Fylde College\\
University of Lancaster\\
Lancaster LA1 4YF\\
United Kingdom 
\email{g.dales@lancaster.ac.uk}

\medskip

\noindent D.~STRAUSS,\\ 
Department of Pure Mathematics\\ 
University of Leeds\\  
Leeds, LS2 9JT\\ 
United Kingdom
\email{d.strauss@leeds.ac.uk}

\medskip

\noindent Y.~ZELENYUK, and Yu.~ZELENYUK\\ 
School of Mathematics,\\
University of Witwatersrand,\\
Wits, 2050, \\
Johannesburg,\\
South Africa
\email{Yevhen.Zelenyuk@wits.ac.za}\\
\email{Yuliya.Zelenyuk@wits.ac.za}

\end{document}